\documentclass[12pt]{amsart}

\usepackage{verbatim}

\usepackage{amsthm}
\usepackage{amsfonts}
\usepackage{amsmath}
\usepackage{amssymb}
\usepackage{verbatim, color, enumerate}

\usepackage{pictex}
\usepackage{url}

\begin{document}
 \newcounter{thlistctr}
 \newenvironment{thlist}{\
 \begin{list}%
 {\alph{thlistctr}}%
 {\setlength{\labelwidth}{2ex}%
 \setlength{\labelsep}{1ex}%
 \setlength{\leftmargin}{6ex}%
 \renewcommand{\makelabel}[1]{\makebox[\labelwidth][r]{\rm (##1)}}%
 \usecounter{thlistctr}}}%
 {\end{list}}

\thispagestyle{empty}

\newtheorem{Lemma}{\bf LEMMA}[section]
\newtheorem{Theorem}[Lemma]{\bf THEOREM}
\newtheorem{Claim}[Lemma]{\bf CLAIM}
\newtheorem{Corollary}[Lemma]{\bf COROLLARY}
\newtheorem{Proposition}[Lemma]{\bf PROPOSITION}
\newtheorem{Example}[Lemma]{\bf EXAMPLE}
\newtheorem{Fact}[Lemma]{\bf FACT}
\newtheorem{definition}[Lemma]{\bf DEFINITION}
\newtheorem{Notation}[Lemma]{\bf NOTATION}
\newtheorem{remark}[Lemma]{\bf REMARK}

\newcommand{\restrict}{\mbox{$\mid\hspace{-1.1mm}\grave{}$}}
\newcommand{\covers}{\mbox{$>\hspace{-2.0mm}-{}$}}
\newcommand{\covered}{\mbox{$-\hspace{-2.0mm}<{}$}}
\newcommand{\notcover}{\mbox{$>\hspace{-2.0mm}\not -{}$}}

\newcommand{\boldalpha}{\mbox{\boldmath $\alpha$}}
\newcommand{\boldbeta}{\mbox{\boldmath $\beta$}}
\newcommand{\boldgamma}{\mbox{\boldmath $\gamma$}}
\newcommand{\boldxi}{\mbox{\boldmath $\xi$}}
\newcommand{\boldlambda}{\mbox{\boldmath $\lambda$}}
\newcommand{\boldmu}{\mbox{\boldmath $\mu$}}

\newcommand{\barzero}{\bar{0}}

\newcommand{\sfq}{{\sf q}}
\newcommand{\sfe}{{\sf e}}
\newcommand{\sfk}{{\sf k}}
\newcommand{\sfr}{{\sf r}}
\newcommand{\sfc}{{\sf c}}
\newcommand{\restr}{\negmedspace\upharpoonright\negmedspace}

\title[Dually Quasi-De Morgan Semi-Heyting Algebras]{JI-distributive, Dually 
Quasi-De Morgan 
Semi-Heyting and Heyting Algebras}              
\dedicatory{Dedicated to Professor P.N. Shivakumar\\ 
A Great Humanitarian who changed the course of my life}
\author{Hanamantagouda P. Sankappanavar}

\date{\today}

\begin{abstract}
The variety $\mathbf{DQD}$ of semi-Heyting algebras  
with a weak negation, called 
{\it dually quasi-De Morgan operation},
and several of its subvarieties 
were investigated in the series ~\cite{Sa11}, 
~\cite{Sa14}, ~\cite{Sa14a}, and ~\cite{Sa15}.    
In this paper we define and investigate a new subvariety 
$\mathbf {JID}$ of $\mathbf{DQD}$, 
called ``JI-distributive, dually quasi-De Morgan semi-Heyting algebras'', 
defined by the identity:  $x' \lor (y \to z) \approx (x' \lor y) \to (x' \lor z)$, as well as the (closely related) variety $\mathbf{DSt}$ of dually Stone semi-Heyting algebras.   
We first prove that $\mathbf {DSt}$ and $\mathbf {JID}$ are discriminator varieties of level 1 and level 2 (introduced in \cite{Sa11}) respectively.   
Secondly, we give a characterization of subdirectly irreducible algebras of the subvariety $\mathbf {JID_1}$ of level 1 of $\mathbf {JID}$. 
 As a first application of it, we derive that the variety $\mathbf {JID_1}$  
 is the join of the variety $\mathbf {DSt}$  
 and the variety 
of De Morgan Boolean semi-Heyting algebras.  As a second application, 
we give a concrete description of the subdirectly irreducible algebras    
in the subvariety $\mathbf {JIDL_1}$ of $\mathbf {JID_1}$ defined by the linear identity:  $(x \to y) \lor (y \to x) \approx 1,$ 
and deduce that   
 the variety $\mathbf {JIDL_1}$ is the join of the variety $\mathbf {DStHC}$ generated by the dually Stone Heyting chains   
 and the variety generated by the 4-element De Morgan Boolean Heyting algebra.   
 Several applications of this result are also given, including   
  a description of the lattice of subvarieties of $\mathbf {JIDL_1}$,
  equational bases of all subvarieties of $\mathbf {JIDL_1}$, and 
  the amalgamation property of all subvarieties of $\mathbf {DStHC}$.  
\end{abstract}

\maketitle

\thispagestyle{empty}

\section{{\bf Introduction}} \label{SA}

The De Morgan (strong) negation and the pseudocomplement  are two of the fairly well known  negations that generalize the classical negation.  
A common generalization of these two negations led to a new variety of algebras, called ``semi-De Morgan algebras'' and a subvariety called ``(upper) quasi-De Morgan algebras'' in \cite{Sa87a}. 

Recently, it is shown, in \cite{LiGrPa17} and \cite{MaLi17}, that semi-De Morgan algebras form an algebraic semantics for a propositional logic, called ``semi-De Morgan logic'', with a Hilbert style axiomatization and a sequent calculus, and with a display calculus, respectively.

In a different vein, as an abstraction of Heyting algebras, semi-Heyting algebras were introduced in  \cite{Sa08}.  It is well known that Heyting algebras are an algebraic semantics for intuitionistic logic.   Recently, in 2013, Cornjo, in \cite{Co11} (see also \cite{CoVi15} for a simpler axiomatization), has presented a logic called ``semi-intutionistic logic'' having semi-Heyting algebras as its algebraic semantics, and having intuitinistic logic as an extension so that the lattice of intermediate intuitionistic logics is an interval in the lattice of ``intermediate'' semi-intuitionistic logics.

Using the dual version of quasi-De Morgan negation, an expansion of semi-Heyting algebras, called   
``dually quasi-De Morgan semi-Heyting algebras ($\mathbf {DQD}$, for short)'' was defined and investigated in \cite{Sa11}, as a common generalization of De Morgan (or symmetric) Heyting algebras \cite{Sa87} (see also \cite{Mo80}) and dually pseudocomplemented Heyting algebras \cite{Sa85}, the latter is recently shown to be an algebraic semantics for a paraconsistent logic by   
Castiglioni and Biraben in \cite{CaBi13}.  In 1942, Moisil \cite{Mo42} has introduced a propositional logic called ``modal symmetric logic'' and, in 1980, Monteiro \cite{Mo80} has shown that De Morgan Heyting algebras are an algebraic semantics for Moisil's modal symmetric logic.      
It should also be mentioned here that the present authors have proposed recently a new propositional logic, called ``De Morgan semi-Heyting logic'',  in \cite{CoSa17} which has De Morgan semi-Heyting algebras as an algebraic semantics and the modal symmetric logic of Moisil 
as an extension.   
 
Several new subvarieties of $\mathbf{DQD}$ were studied in  ~\cite{Sa11}, ~\cite{Sa14}, ~\cite{Sa14a}, and ~\cite{Sa15}, including   
the vareity $\mathbf {DStHC}$ generated by the dually Stone Heyting chains (i.e., the expansion of the G\"{o}del variety by the dual Stone operation), the variety $\mathbf {DMB}$ of De Morgan Boolean semi-Heyting algebras 
and, in particular, the variety  
$\mathbf {DMBH}$ generated by the $4$-element De Morgan Boolean Heyting algebra.   
These investigations led us naturally to the problem of equational axiomatization for the join of the variety  $\mathbf {DStHC}$ and the variety $\mathbf {DMBH}$.  Our investigations into this problem led us to the results of the present paper that include a solution to the just mentioned problem.

In this paper we define and investigate a new subvariety 
$\mathbf {JID}$ of $\mathbf{DQD}$, 
called ``JI-distributive, dually quasi-De Morgan semi-Heyting algebras'', 
defined by the identity:  $x' \lor (y \to z) \approx (x' \lor y) \to (x' \lor z)$, as well as the (closely related) 
 variety $\mathbf{DSt}$ of dually Stone semi-Heyting algebras.   
We first prove that $\mathbf {DSt}$ and $\mathbf {JID}$ are discriminator varieties of level 1 and level 2 (see Section 2 for definition) respectively.  
Secondly, we prove that the lattice of subvarieties of $\mathbf {DStHC}$ is an $\omega +1$-chain.
Thirdly, we give a characterization of subdirectly irreducible algebras of the subvariety $\mathbf {JID_1}$ of level 1. 
As a first application of it, we derive that the variety $\mathbf {JID_1}$ 
 is the join of the variety $\mathbf {DSt}$ and the variety $\mathbf{DMB}$.    
As a second application, we give a concrete description of the subdirectly irreducible algebras   
in the subvariety $\mathbf {JIDL_1}$ of $\mathbf {JID_1}$ defined by the linear identity:  $(x \to y) \lor (y \to x) \approx 1,$ and deduce that   
 the variety $\mathbf {JIDL_1}$ is the join of the variety $\mathbf {DStHC}$ generated by the dually Stone Heyting chains and $\mathbf{DMBH}$.  
Other applications include   
  a description of the lattice of subvarieties of $\mathbf {JIDL_1}$,
  equational bases of all subvarieties of $\mathbf {JIDL_1}$, and 
  the amalgamation property of all subvarieties of $\mathbf {DStHC}$.

More explicitly, the paper is organized as follows: 
In Section 2 we recall definitions, notations and results 
from \cite{Sa11}, \cite{Sa14} and \cite{Sa14a} and also prove some new results 
needed in the rest of the paper.  In Section 3, we define the variety $\mathbf {JID}$ of JI-distributive, dually quasi-De Morgan semi-Heyting algebras 
and give some arithmetical properties of $\mathbf {JID}$.   
In particular, we show that $\mathbf {JID}$ satisfies the $\lor$-De Morgan law and  
the level $2$ identity: $(x \land x'^*){^{2\rm({'^*}\rm)}} \approx (x \land x'^*){^{3\rm({'^*}\rm)}}$
-- these two propertes allow us to apply \cite[Corollary 8.2(a)]{Sa11} to deduce that $\mathbf{JID}$ is a discriminator variety.  These properties also play a crucial role in the rest of the paper. 
Section 4 will prove that the variety $\mathbf{DSt}$ is a discriminator variety of level 1.  It will also present some properties of $\mathbf{DSt}$, which, besides being of interest in their own right, will also be useful in the later sections.  It is also proved that the lattice of subvarieties of $\mathbf {DStHC}$ is an $\omega +1$-chain. 
In Section 5, we give a characterization of subdirectly irreducible algebras in the subvariety $\mathbf {JID_1}$ of $\mathbf {JID}$ and deduce that $\mathbf {JID_1}$ is the join of $\mathbf {DSt}$ and the variety of De Morgan Boolean semi-Heyting algebras.   
Several applications of this characterization are given in sections 7 and 8.
We investigate, in Section 6, the variety $\mathbf {JIDL_1}$ of JI-distributive dually quasi-De Morgan linear semi-Heyting algebras of level 1.   
An explicit description of subdirectly irreducible algebras in $\mathbf {JIDL_1}$ is given, and from this description it is deduced that $\mathbf{JIDL_1} =  \mathbf{DStHC}  \lor \mathbf{V(D_2)}$.  
Several applications of this result are given in Section 7.
It is shown that the lattice of subvarieties of $\mathbf {JIDL_1}$ is isomorphic to $\mathbf{1 \oplus [(\omega + 1)} \times \mathbf {2}]$, where $\mathbf {2}$ is the 2-element chain.  
Also, (small) equational bases for all subvarieties of $\mathbf {JIDL_1}$ are given.   
Finally, it is shown that all subvarieties of 
$\mathbf {DStHC}$ have the amalgamation property.

\vspace{1cm}
\section{\bf {Preliminaries}} \label{SB}

In this section we recall some notions and known results needed to make this paper as self-contained as possible.  However, for more basic information, we refer the reader to  \cite{BaDw74}, \cite{BuSa81} and \cite{Ra74}.

An algebra ${\mathbf L}= \langle L, \vee ,\wedge ,\to,0,1 \rangle$
is a {\it semi-Heyting algebra} (\cite{Sa08}) if \\
 $\langle L,\vee ,\wedge ,0,1 \rangle$ is a bounded lattice and ${\mathbf L}$ satisfies:
\begin{enumerate}
\item[{\rm(SH1)}] $x \wedge (x \to y) \approx x \wedge y$,
\item[{\rm(SH2)}] $x \wedge(y  \to z) \approx x \wedge [(x \wedge y) \to (x \wedge z)]$,
\item[{\rm(SH3)}] $x \to x \approx 1$.
\end{enumerate}
Semi-Heyting algebras are distributive and pseudocomplemented,
with $a^* := a \to 0$ as the pseudocomplement of an element $a$.
 
Let ${\mathbf L}$ be a semi-Heyting algebra.   
${\mathbf L}$ is a {\it Heyting algebra} if ${\mathbf L}$ satisfies:
\begin{enumerate}
\item[{\rm(H)}] $(x \wedge y) \to y \approx 1$.
\end{enumerate}
$\mathbf{L}$ is a {\it Boolean semi-Heyting algebra} 
if ${\mathbf L}$ satisfies:
\begin{enumerate}
\item[{\rm(Bo)}] $x \lor x^{*} \approx 1$.
\end{enumerate}
$\mathbf{L}$ is a {\it Boolean Heyting algebra} 
 if ${\mathbf L}$ is a Heyting algebra that satisfies (Bo).

The following definition, taken from \cite{Sa11}, is central to this paper.
\begin{definition}
An algebra ${\mathbf L}= \langle L, \vee ,\wedge ,\to, ', 0,1
\rangle $ is a {\it semi-Heyting algebra with a dual quasi-De Morgan
operation} or {\it dually quasi-De Morgan semi-Heyting algebra}
\rm{(}$\mathbf {DQD}$-algebra, for short\rm{)}  if\\
 $\langle L, \vee ,\wedge ,\to, 0,1 \rangle $ is a semi-Heyting algebra, and
${\mathbf L}$ satisfies:
 \begin{itemize}
    \item[(a)]  $0' \approx 1$ and $1' \approx 0$,
     \item[(b)] 
      $(x \land y)' \approx x' \lor y'$,
    \item[(c)]
    $(x \lor y)''  \approx x'' \lor y''$,     
    \item[(d)]  
   $x'' \leq x$.
\end{itemize}
\noindent 
Let $\mathbf{L}$ be a $\mathbf {DQD}$-algebra.  ${\bf L}$ is a {\it dually pseudocomplemented semi-Heyting algebra} {\rm(}$\mathbf{DPC}$-algebra{\rm)} if ${\bf
L}$ satisfies: 
\begin{itemize}
\item [(e)] $x \lor x' \approx 1$.
\end{itemize}
$\mathbf {L}$ is a {\it dually Stone semi-Heyting algebra}   
  {\rm(}$\mathbf{DSt}$-algebra\rm) if ${\bf L}$ satisfies the dual Stone identity:
\begin{itemize}
\item[(DSt)]  $x' \land x'' \approx 0$.
\end{itemize}
It should be noted that if (DSt) holds in a $\mathbf {DQD}$-algebra ${\mathbf L}$, then (e) holds in ${\mathbf L}$ as well, and hence $'$ is indeed the dual pseudocomplement 
satisfying the Stone identity, and so 
${\mathbf L}$ is indeed a $\mathbf {DSt}$-algebra. \\
 $\mathbf {L}$ is a {\it De Morgan semi-Heyting algebra}   
  {\rm(}$\mathbf{DM}$-algebra{\rm)} if ${\bf
L}$ satisfies:
\begin{itemize}
\item[(DM)]  $x'' \approx x$.
\end{itemize}
\end{definition}

The varieties of $\mathbf {DQD}$-algebras, $\mathbf{DPC}$-algebras, $\mathbf{DSt}$-algebras,  
 $\mathbf{DM}$-algebras 
 are denoted, respectively, by $\mathbf {DQD}$,
$\mathbf{DPC}$, $\mathbf{DSt}$, and $\mathbf{DM}$.    
If the underlying semi-Heyting algebra of a $\mathbf {DQD}$-algebra is a Heyting
algebra, then we add    
``$\mathbf{H}$'' at the end of the names of the varieties that will be considered in the sequel.
Thus, for example, $\mathbf{DStH}$ denotes the variety of dually Stone Heyting algebras.

The following lemmas are basic to this paper.  The proof of the first lemma is straightforward 
and is left to the reader.

\begin{Lemma} \label{2.2}
Let ${\mathbf L} \in \mathbf{DQD}$ and let $x,y, z \in L$.  Then
\begin{enumerate}[\rm(i)]
 \item  $1'^{*}=1$, and $1 \to x =x$,
\item  $x \leq y$ implies $x' \geq y'$,  
\item $(x \land y)'^{*}=x'^{*} \land y'^{*}$,
\item $ x''' = x'$, \label{v} 
\item $(x \lor y)' = (x'' \lor y)'$, \label{vii}         
\item $x \land [y \lor (x \to z)] = x \land (y \lor z)$, \label{xii} 
\item  $x \land (x \to y)'' \leq y$. \label{xix}  
\end{enumerate}
\end{Lemma}

\begin{Lemma} \label{2.3}
Let $L \in \mathbf{DQD}$ and  
$x,y \in L$.  Then
\begin{itemize}
\item[(1)] \label{Lemma_72}     
$(x \lor y)' \leq  x' \to (x \lor y)'$,   
\item[(2)] \label{Lemma_84}  
$[x \lor (y \lor z)']' = (x \lor y')'  \lor (x \lor z')' $,  
\item[(3)] \label{Lemma_3.70}   
$x \land [(x \to y) \lor z] = x \land (y \lor z)$,  
\item[(4)] \label{Lemma_4.62} 
$y \land [x \to (y \land z)] = y \land (x \to  z)$,    
\item[(5)] \label{Lemma_4.75}  
$x \to (y  \land z) \  \geq \ y \land (x \to  z)$,  
\item[(6)] \label{Lemma_4.93} 
$x \   \leq \ y \to (x \land y)$,  
\item[(7)]  
$(x \lor y)' = x' \land [(x  \lor y)' \lor \{x'  \to (x \lor y)'\}'' ] $,  
\item[(8)] $x \leq (x \to y) \to y$.   
\end{itemize}
\end{Lemma}

\begin{proof} 
\ \\
 (1) is straightforward to verify since $(x \lor y)' \leq x'$.  \\
 
\noindent (2):  
$\begin{aligned} [t]
[(x \lor (y \lor z)']' 
& = [x'' \lor (y \lor z)']'  &\text{by Lemma \ref{2.2} (\ref{vii})} \\      
&= [x' \land (y \lor z)]''  \\
&=[(x' \land y) \lor (x' \land z)]''\\
&= (x' \land y)'' \lor  (x' \land z)''\\
&= (x'' \lor y')' \lor (x'' \lor z' )'\\
&= (x \lor y')' \lor (x \lor z' )' &\text{by Lemma \ref{2.2} (\ref{vii})}\\  
\end{aligned}$

\noindent(3) and (4) are easy to verify.\\

\noindent (5):  
$\begin{aligned} [t]
[x \to (y  \land z)] \land y \land (x \to  z)
&= y \land [x  \to (y \land z)] \land  (x \to  z)\\
&= y  \land  (x \to  z) \text{   by (4)}. 
\end{aligned}$\\

\noindent (6): 
$x = x  \land (y \to y)
\leq y \to  (x \land y)$ by (5). \\

\noindent (7): 
$\begin{alignedat}[t]{2}
(x \lor y)'
& = x' \land (x \lor y)' \\
&= x' \land [x' \to (x \lor y)'] \\
&= x' \land [\{x' \to (x \lor y)'\} \lor  \{x' \to (x \lor y)'\}''] \\  
&= [x' \land (x \lor y)'] \lor [x' \land \{x' \to (x \lor y)' \}''] \\
&= x' \land [(x \lor y)' \lor  \{x'  \to (x \lor y)' \}''],  
\end{alignedat}$.\\

\noindent (8): 
 $x  \land  [(x \to y) \to y] = x \land [\{x \land (x \to y)\} \to (x \land y)] = x \land [(x \land y) \to (x \land y)] =x \land 1 =x $, 
completing the proof.
\end{proof}

The following three $4$-element algebras, called $\mathbf{D_1}$, $\mathbf{D_2}$, and $\mathbf{D_3}$ (following the notation of \cite{Sa11}), in $\mathbf{DQD}$, play an important role in the sequel.
All three of them have the Boolean lattice reduct 
with the universe $\{0,a,b,1\}$, where $b$ is the Boolean complement of $a$,    
and the operation $'$ is defined as follows: 
$a' = a, \ b' = b, \ 0'=1, \ 1'=0,$ while the operation $\to$ is defined in Figure 1.  \\
\ \\ \ \\

\setlength{\unitlength}{.7cm} \vspace*{.1cm}        
\begin{picture}(7,4)

\put(.4,2.5){$\bf D_1$ \ :}

\put(7.6,2.5){$\bf D_2$ \ :}

\put(1.5,2.4){\begin{tabular}{c|cccc}                          
 $\to$ \ & \ 0 \ & \ $1$ \ & \ $a$\ &\ $b$ \\ \hline
  0 & 1 & 0 & $b$ & $a$ \\
  $1$ & 0 & 1 & $a$ & $b$ \\
  $a$ & $b$ & $a$ & 1 &0\\
  $b$ & $a$ & $b$ & 0 & 1\\
\end{tabular}
}
\put(8.6,2.4){\begin{tabular}{c|cccc}
 $\to$ \ & \ 0 \ & \ $1$ \ & \ $a$\ &\ $b$ \\ \hline
  0 & 1 & 1 & 1 & 1 \\
  $1$ & 0 & 1 & $a$ & $b$ \\
  a & $b$ & $1$ & 1 & $b$ \\
  $b$ & $a$ & 1 & $a$ & 1\\
\end{tabular}
}
                          
 \end{picture}  
\medskip
\ \\ \ \\
\setlength{\unitlength}{.7cm} \vspace*{.1cm}    

 \begin{picture}(7, 3) \\

 \put(1.3,2.0){$\bf D_3$ \ :}                   
                            
\put(2.6,2.4){\begin{tabular}{c|cccc}                  
 $\to$ \ & \ 0 \ & \ $1$ \ & \ $a$\ &\ $b$ \\ \hline
  0 & 1 & $a$ & 1 & $a$ \\
  $1$ & 0 & 1 & $a$ & $b$ \\
  $a$ & $b$ & $a$ & 1 &0\\
  $b$ & $a$ & 1 & $a$ & 1\\
\end{tabular}
}
\vspace{.5cm}
\put(6,0.1){Figure 1}    
\end{picture}

\vspace{.7cm}

Let $\mathbf {DQB}$ and $\mathbf {DMB}$ denote respectively the subvarieties of $\mathbf {DQD}$  and $\mathbf {DM}$ defined by (Bo).  
We also note $\mathbf {DQBH}$ and $\mathbf {DMBH}$ denote respectively the subvariety of $\mathbf {DQB}$ and $\mathbf {DMB}$ defined by (H). 
$\mathbf {V(K)}$ denotes the variety generated by the class K of algebras in $\mathbf {DQD}$. 
The following proposition is proved in  (\cite{Sa11}) and is needed later in this paper.
\begin{Proposition} \label{Prop_DQB}         
\begin{thlist}  
\item[a]  $\mathbf{DQB} = \mathbf{DMB} =  \mathbf{V(\{D_1, D_2, D_3\}})$, 
\item[b] $\mathbf{DQBH} = \mathbf{DMBH} =  \mathbf{V( D_2})$.
\end{thlist}  
\end{Proposition}
The following definition is from \cite{Sa11}.

\begin{definition}\label{5.5}

Let $\mathbf{L} \in \mathbf{DQD}$ and $x \in {\bf L}$.  For $n \in \omega$, we
define $t_n(x)$ recursively as follows:\par

\begin{center}
$x{^{0(}{'{^{*)}}}} := x$;  \\
$x^{(n+1)(}{'{^{*)}}} := (x{^{n(}}{'{^{*)}}})'{^*}$, for $n \geq 0$;\\
$t_0(x) := x$, \\
     $t_{n+1}(x) := t_n(x) \land x^{(n+1){\rm(}}{'^{*{\rm{)}}}}$, for $n \geq 0$.
\end{center}

Let $n \in \omega$. 
The subvariety $\mathbf {DQD_n}$ {\it of level $n$}
of $\mathbf {DQD}$ is defined by the identity:
\begin{equation} 
t_n(x)\approx t_{(n+1)}(x); \tag{lev\ $n$}
\end{equation}
For a subvariety $\mathbf {V}$ of $\mathbf {DQD}$, we let $\mathbf {V_n}:= \mathbf {V} \cap \mathbf {DQD_n}$.  
\end{definition}

Recall from \cite{Sa11} (or \cite{Sa14})  that $\mathbf{BDQDSH}$ is the subvariety of $\mathbf{DQD} \ (=\mathbf{DQDSH})$ defined by the identity:\\
(B) $(x \lor y^*)' \approx x' \land y{^*}'$.  \\
We will abbreviate $\mathbf{BDQDSH}$ by $\mathbf{BDQD}$.

Next theorem, which was proved in \cite[Corollary 4.1]{Sa14}
(which is, in turn, a consequence of
Corollaries 7.6 and 7.7 of  \cite{Sa11}), will play a fundamental role in this paper.

The following ``simplicity' condition'', (SC), is crucial in the rest of the paper.
\begin{itemize}
\item[(SC)] \quad For every $x \in L$, if $x \neq 1$, then $x \land x'{^*} =0$.
\end{itemize}

\begin{Theorem} \cite[Corollary 4.1]{Sa14}  \label{Theorem_BDQD}
Let $n \in \mathbb N$ and $\mathbf{L} \in \mathbf {BDQD_1}$   
with $|L| \geq 2$.  Then the following are equivalent:
\begin{enumerate}
\item[(1)] ${\mathbf L}$ is simple,
\item[(2)]  $\mathbf{L}$ is subdirectly irreducible,
\item[(3)] $\mathbf{L}$ satisfies {\rm (SC)}.
\end{enumerate}
\end{Theorem}

\vspace{1cm}
\section{JI-distributive, dually quasi-De Morgan semi-Heyting algebras}    
 
The identity,  $x  \lor (y \to z) \approx (x \lor y) \to (x \lor z)$, was shown in \cite[Corollary 3.55]{Sa14a} to be a base for the variety generated by $\mathbf{D_2}$, relative to $\mathbf{DQD}$.   Let us refer to this identity as ``strong JI-distributive identity''.  
We now introduce a slightly weaker identity, called ``JI-distributive identity'' (by restricting the first variable to ``primed'' elements).  The variety defined by this identity is the subject of our investigation in the rest of this paper.

\begin{definition}
The subvariety $\mathbf {JID}$ of $\mathbf {DQD}$ is defined by:
\begin{enumerate}
 \item[{\rm(JID)}]  $x'  \lor (y \to z) \approx (x' \lor y) \to (x' \lor z)$ 
  {\rm ({\bf R}estricted distribution of 
 \item[ ] 
\hspace{6.5cm} {\bf J}oin over {\bf{I}}mplication\rm)}.
\end{enumerate}
\end{definition}

Examples of $\mathbf {JID}$ come from an interesting source to which we shall now turn. 
But, first we need some notation.  

A $\mathbf{DQD}$-algebra is a $\mathbf{DQD}$-chain if its lattice reduct is a chain.  
$\mathbf{DQDC}$ [$\mathbf{DPCC}$] denotes the variety generated by the $\mathbf{DQD}$-chains
[$\mathbf{DPC}$-chains]. 
The following lemma provides an important class of examples of $\mathbf {JID}$, which is partly the motivation for our interest in $\mathbf {JID}$. 

\begin{Lemma} \label{DPCSHC_JID}
$\mathbf {DPCC}  \models {\rm (JID)}$.  Hence, $\mathbf {DPCC} \subseteq \mathbf {JID}.$ 
\end{Lemma}

\begin{proof}
Let $\mathbf{A}$ be a $\mathbf{DPC}$-chain and let $a \in A \setminus \{1\}$.   
 Since $\mathbf{A}$ is a chain, we have $a' \leq a$ or $a \leq a'$ , from which we get that 
$a \lor a' \leq a$ \ \text{or} \  $a \lor a' \leq a'.$
Since $\mathbf {A}$ is dually pseudocomplemented, we have $a \lor a' = 1$, implying $a'=1$.  Now, 
it is routine to verify (JID) holds in $\mathbf {A}$.   
\end{proof}

For $\mathbf{L}$ a {$\mathbf{DPC}$-chain, it was observed in the proof of the preceding lemma  that the dual pseudocomplement $'$ satisfies: $a' = 1$, if $a \neq 1$,  
and hence $\mathbf{L} \models {\rm(DSt)}$.  Thus, we have the following corollary.  Note that 
$\mathbf{DStC}$ is the variety generated by the dually Stone semi-Heyting chains. 

\begin{Corollary} \label{Cor_DStC}  
$\mathbf{DPCC} = \mathbf{DStC}$.  
\end{Corollary}
 
From now on, we use 
$\mathbf{DPCC}$ and $\mathbf{DStC}$ interchangeably.
Observe also that  $\mathbf{D_1}$, $\mathbf{D_2}$, and $\mathbf{D_3}$ are examples of $\mathbf{JID}$-algebras.

In the rest of this section we present several useful arithmetical properties of $\mathbf {JID}$.
Following our convention made earlier, $\mathbf{JIDH}$ denotes the subvariety of $\mathbf {JID}$ defined by the identity (H).

{\bf Throughout this section, we assume that $\mathbf{L} \in \mathbf{JID}$.}  

\begin{Lemma} \label{L1} Let $x,y,z \in \mathbf{L}$.  Then 
\begin{enumerate} 
\item    $x' \to (x' \lor y) = x' \lor (x' \to y)$,  
\item   $x' \to (x' \lor y) =  x' \lor (0 \to y)$,  
\item   $x' \lor (x' \to y) = x' \lor (0 \to y)$; in particular, $x' \lor x'^* = 1$, \label{107}  
\item  $(x' \lor y) \to x'  = x' \lor y^*$,  \label{111}  
\item  $(x' \lor y) \to x' = x' \lor (y \to x')$, \label{103}   
\item  $x' \lor (y \to x') = x' \lor y^*$, \label{112}  
\item $x' \to (x \lor y)'  = x'^* \lor (x \lor y)' $. \label{7} 
\end{enumerate}
\end{Lemma}

\begin{proof}
Observe that
$x' \to (x' \lor y) = (x' \lor x') \to (x' \lor y)   
                         = x' \lor (x' \to y)$  by  (JID),
which proves (1).   
To prove (2), again using (JID), we get
$x' \lor (0 \to y)   =  (x' \lor 0) \to (x' \lor y) 
                         = x' \to (x' \lor y)$.                                                     
(3) is immediate from (1) and (2).  
For (4),
$(x' \lor y) \to x'  = (x' \lor y) \to (x' \lor 0)
                          = x' \lor (y \to 0) 
                          = x' \lor y^*$, in view of (JID).
Next,
$(x' \lor y) \to x' 
 = (x' \lor y) \to (x' \lor x')
 =  x' \lor (y \to x')$, proving (5), and
(6) is immediate from (4) and (5). 
For (7),  we have
\begin{eqnarray*}
x' \to (x \lor y)' 
&=& (x \lor y)' \lor [x' \to (x \lor y)']  \text{ by Lemma \ref{2.3} (1)}\\ 
& =& (x \lor y)' \lor x'^* \text{ by (6)}.  
\end{eqnarray*}
\end{proof}

We now prove an important property of the variety $\mathbf{JID}$, namely the $\lor$-De Morgan law.
We denote by $\mathbf{Dms}$ the subvariety of $\mathbf{DQD}$ (called ``dually ms semi-Heyting algebras'') defined by
\begin{center}
  $(x \lor y)' \approx x' \land y'$ \qquad {\rm($\lor$-De Morgan Law)}.
\end{center}

\begin{Theorem} \label{Theorem_105} 
\ $\mathbf{JID} \subseteq \mathbf{Dms}$.  
\end{Theorem}

\begin{proof} Let $x,y  \in \mathbf {L}$.  As $x' \land x'{^*}'' \leq x' \land x'^*=0$, we get
$x' \land y' 
= (x' \land x'{^*}'') \lor (x' \land y')$.  Hence, 
$$
\begin{array}{lcll}
x' \land y' 
&=& x' \land (x'{^*}'' \lor y') \\ 
&=& x' \land [(x \lor y)' \lor x'{^*}'' \lor y']  &\text{since $(x \lor y)' \leq y'$} \\
&=& x' \land [(x \lor y)' \lor x'{^*}'' \lor y''']  &\\
&=& x' \land [(x \lor y)' \lor (x'^* \lor y')'']   &\\
&=& x' \land [(x \lor y)' \lor \{(x'^* \lor x')' \lor (x'^* \lor y')' \}']  &\text{by Lemma \ref{L1} (\ref{107}}) \\ 
&=& x' \land [(x \lor y)' \lor \{x'^* \lor (x \lor y)' \}'']  &\text{by Lemma \ref{Lemma_84} (2)} \\      
&=& x' \land [(x \lor y)' \lor \{x' \to (x \lor y)' \}'']  &\text{by Lemma \ref{L1} (7)}\\      
&=&  (x \lor y)'  &\text{by Lemma \ref{2.3} (7)}.  
\end{array}
$$
Hence, $\mathbf{JID} \subseteq \mathbf{Dms}$.
\end{proof}

The following lemma is useful in this and later sections.

\begin{Lemma} \label{E} Let $x,y,z \in \mathbf{L}$.  Then  
\begin{enumerate}
\item[(1)]   $x'{^*}''  =  x'^*$,   
\item[(2)]   $x''{^*}  =  x'{^*}'$,  
\item[(3)]  $x \to (x  \land y') = x^* \lor y'$,    
\item[(4)]  $(x \land y'^*)^*  =  y' \lor x^*$,    
\item[(5)]  $(x' \lor y''{^*}){^*}'  =  (x'' \land y'{^*})^*$. 
\end{enumerate}
\end{Lemma}

\begin{proof} 
\noindent(1):  
From Lemma \ref{L1} (3) we have $x' \lor x'^* =1$, which 
yields $x''' \lor x'{^*}'' =1$, implying $x' \lor x'{^*}'' =1$, leading to $x'^* \leq x'{^*}'' $; thus,
$x'^* = x'{^*}'' $.\\

\noindent(2):   
From $x' \lor x'^* =1$ and Theorem \ref{Theorem_105} 
we get  $x'' \land x'{^*}' =0$, implying $x'{^*}' \leq x''^*$.  To prove the reverse inequality, from  
$x' \land x'^* =0$,   
we get $x'' \lor x'{^*}' =1$, from which it follows that
$x''^* \leq x'{^*}' $. \\

\noindent(3): 
$\begin{alignedat}[t]{2}
x^* \lor y'
&= (y' \lor x)  \to y'  \quad &\text{ by (JID)} \\  
&= (y'  \lor x) \to [y' \lor (x \land  y')]  &\\  
&= y'  \lor [x \to (x \land  y')]  \quad &\text{by (JID)} \\                                                    
&= x \to (x  \land y')  \quad &\text{by Lemma \ref{2.3} (6)}. \\   
\end{alignedat}$\\

\noindent(4):   
$\begin{alignedat}[t]{2}
(x \land y'^*)^*
&= (x \land  y'''{^*})^*   \quad &\text{ by   Lemma \ref{2.2} (v)}  \\                  
&=  (x \land y'{^*}'')^*   \quad &\text{  by (2) (twice)}\\ 
&=  (x \land y'{^*}'')   \to (y' \land x \land y'{^*}'')   \quad &\text{  as $y'  \land y'{^*}''=0$}\\ 
&=  y' \lor (x \land y'{^*}'')^*  \quad &\text{ by (3)} \\        
&=  y' \lor (x \land y'''{^*})^*   \quad &\text{ by (2) (twice)}\\ 
&=  y' \lor (x \land y'{^*})^*  \quad &\text{ by Lemma \ref{2.2} (v)}\\ 
&= y' \lor [(x \land y'^*) \to 0] \\ 
&=  [y' \lor (x \land y'^*)] \to y'    \quad &\text{ by (JID)}\\
&=  [(y' \lor x) \land (y' \lor y'^*)] \to y' \\ 
&=  [(y' \lor x) \land 1] \to y'  \quad & \text{ by Lemma \ref{L1} (3)}\\
&=  (y' \lor x)  \to y' \\
&=  y' \lor (x  \to 0)  \quad &\text{ by (JID)}\\
&= y' \lor x^*. \\
\end{alignedat}$\\

\noindent(5): 
$\begin{alignedat}[t]{2}
(x' \lor y''{^*}){^*}' 
&=  (x' \lor y'{^*}'){^*}' \quad &\text{ by (2)}\\ 
&=  (x \land y'^*)'{^*}' \\
&=  (x \land y'^*)''^* \quad &\text{ by (2)}\\
&=  (x' \lor y'{^*}')'{^*} \\
&=  (x'' \land y'{^*}'')^* \quad &\text{  by Theorem \ref{Theorem_105}} \\
&= (x'' \land y'''{^*})^* \quad &\text{ by (2)  (twice)} \\  
&= (x'' \land y'{^*})^*.
\end{alignedat}$\\
This completes the proof.
\end{proof}

\subsection{An Alternate Definition of ``level\ $n$''.}

\ \\ \ \\ 
The following lemmas enable us to give an alternate definition of ``Level $n$''.  

\begin{Lemma} \label{Lemma_for_altrnate-definition}  
Let $x \in \mathbf{L}$.  Then $x'^{**} = x'$.  
\end{Lemma}

\begin{proof} 
Since $x' \lor x'^* =1$ by Lemma \ref{L1}, and $x' \land x'^* =0$, we get  
$x'^{**} = x'$.
\end{proof}

\begin{Lemma}
Let $x \in \mathbf{L}$.  Then $x \land x'^* \land x'{^*}'^* =  (x \land x'^*)'^*$.   
\end{Lemma}

\begin{proof}
$\begin{alignedat}[t]{2}
x \land x'^* \land x'{^*}'^*   
& = x  \land x'^* \land x''^{**}  \quad &\text{ by Lemma \ref{E} (2)}\\  
& =  x \land x'^* \land x''  \quad &\text{  by Lemma \ref{Lemma_for_altrnate-definition}}\\ 
& =  x'^* \land x'' \\
&= x'^*  \land x''^{**}   \quad &\text{ by Lemma \ref{Lemma_for_altrnate-definition}} \\  
&= x'^*  \land x'{^*}'^* \quad &\text{ by Lemma \ref{E} (2)} \\  
&= (x \land x'^*)'^*.
\end{alignedat}$\\
\end{proof}

Recall $\mathbf{JID_n} = \mathbf{JID} \cap \mathbf{DQD_n}$.  The above lemma allows us to make the following alternate (but equivalent) definition for $\mathbf{JID_n}$, for $n \in \omega$.

\begin{definition}
Let $n \in \omega$.
The variety $\mathbf{JID_n}$ 
is the subvriety of $\mathbf{JID}$ defined by
\begin{itemize}
\item[{\rm(Lev  n)}]  \quad  $(x \land x'^*){^{n(}}{'{^{*)}}} \approx  (x \land x'^*)^{(n+1)(}{'{^{*)}}}. $
\end{itemize}
\end{definition} 
In the rest of the paper we will mostly consider varieties of Level 1 and Level 2.

\subsection{The Level of JID}
\  \\ \ \\
Next, we wish to prove that $\mathbf{JID}$ is at Level $2$.

\begin{Theorem} \label{DmsSH}
$\mathbf{JID} \subseteq \mathbf{Dms_2}$; but $\mathbf{JID} \not \subseteq \mathbf{Dms_1}$.   
\end{Theorem}

\begin{proof} 

We already know from Theorem \ref{Theorem_105} that $\mathbf{JID} \subseteq \mathbf{Dms}$.  We now prove the ``level 2'' identity.
\begin{align*}
(x \land x'^*)'{^*}'{^*} &= (x' \lor x'{^*}'){^*}'{^*}  \\
                                      &= (x' \lor x''{^*}){^*}'{^*} & \mbox{by Lemma \ref{E} (2)}\\ 
                                       &= (x'' \land x'{^*})^{**} & \mbox{by  Lemma \ref{E} (5)}  \\                                         
                                      &= (x' \lor x''{^*})^* & \text{by  Lemma \ref{E} (4)} \\ 
                                      &= (x' \lor x'{^*}')^* & \text{by  Lemma \ref{E} (2)}\\  
                                      &= (x \land x'{^*})'^*.   
\end{align*}
To finish off the proof, we note that $\mathbf{SIX} \in \mathbf{JID}$, but it is not of level 1, where $\mathbf{SIX}$ is the algebra whose lattice reduct, $\to$ and $'$ are given in  
Figure 2.
\end{proof}
\vspace{2cm}
\setlength{\unitlength}{.7cm} 
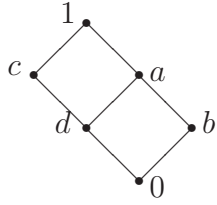
\vspace*{.5cm}
\begin{picture} (8,1)  

\put(4,0){\circle*{.15}}  
\put(3.5,0){$1$}  

\put(5.2,-1.1){$a$}  

\put(3,-1){\circle*{.15}} 
\put(2.5,-1){$c$}
\put(5,-1){\circle*{.15}} 

\put(4,-2){\circle*{.15}}  
\put(3.4, -2.1){$d$}  

\put(6,-2){\circle*{.15}}  
\put(6.2, -2.1){$b$}  

\put(5,-3){\circle*{.15}} 
\put(5.2,-3.3){$0$}

\put(4,-2){\line(1,1){1}}
\put(5,-3){\line(1,1){1}}

\put(4,0){\line(1,-1){1}}
\put(5,-1){\line(1,-1){1}}
\put(4,-2){\line(1,-1){1}}

\put(3,-1){\line(1,1){1}}
\put(3,-1){\line(1,-1){1}}

\put(5,-10.0){Figure 2}      
\end{picture}

\vspace{2cm}

\ \\

\begin{tabular}{r|rrrrrr}
$\to$: & $0$ & $1$ & $a$ & $b$ & $c$ & $d$\\
\hline
    $0$ & $1$ & $1$ & $1$ & $1$ & $1$ & $1$ \\
    $1$ & $0$ & $1$ & $a$ & $b$ & $c$ & $d$ \\
    $a$ & $0$ & $1$ & $1$ & $b$ & $c$ & $c$ \\
   $b$ & $c$ & $1$ & $1$ & $1$ & $c$ & $c$ \\
   $c$ & $b$ & $1$ & $a$ & $b$ & $1$ & $a$ \\
    $d$ & $b$ & $1$ & $1$ & $b$ & $1$ & $1$
\end{tabular} \hspace{.5cm}
\begin{tabular}{r|rrrrrr}
$'$: & $0$ & $1$ & $a$ & $b$ & $c$ & $d$\\
\hline
   & $1$ & $0$ & $b$ & $b$ & $c$ & $1$
\end{tabular} \hspace{.5cm}
\ \\ \ \\

\vspace{.7cm}

The following corollary is immediate from the above theorem and ~\cite[Corollary 8.2(a)]{Sa11}. 

\begin{Corollary} \label{Discr}
 $\mathbf{JID}$ is a discriminator variety of {\rm level 2}. 
\end{Corollary}

\vspace{.8cm}
\section{Dually Stone Semi-Heyting algebras}

The study of dually Stone Heyting algebras goes back to \cite{Sa85}, while the investigations into dually Stone semi-Heyting algebras were initiated in \cite{Sa11}.
 In this section we will prove that the variety $\mathbf{DSt}$ is a discriminator variety of level 1 and  also present some of its properties 
 that, besides being of interest in their own right, will be needed in the later sections.   We will also prove that the lattice of subvarieties of $\mathbf {DStHC}$ is an $\omega + 1$-chain--a result which was implicit in \cite[Section 13]{Sa11}.

See Section 2 for the definition of the condition (SC). 
The following theorem will be useful in the next section.

\begin{Theorem} \label{DSt_JID}  
\begin{thlist}
\item[a] $\mathbf{DSt} \models x'' \approx x'^*$;    
\item[b] $\mathbf{DSt} \models {\rm(Lev\ 1)}$;   
\item[c] If $\mathbf{L} \in \mathbf{DSt}$ and  $\mathbf{L} \models {\rm (SC)}$, then $\mathbf{L} \in \mathbf{JID_1}$.   
\end{thlist}
\end{Theorem}

\begin{proof} 
We note that (a) is the dual of a well known property of Stone algebras.
From (a)   
we have $(x \land x'^*)'^* = (x \land x'')'^* = x'''^* = x'^* = x'' = x \land x'^*$, implying that (b) holds.      
Finally, let $\mathbf{L} \in \mathbf{DSt}$ and satisfy (SC) and let $a \in L \setminus \{1\}$.  Then, by (SC) and (a), we have 
 $a'' = a \land a'' = a \land a'^*=0$, implying $a'=1.$  Then it is straightforward to verify that $\mathbf{L} \models {\rm (JID)}$.  Hence, (c) holds, in view of (b).       
\end{proof}

\begin{remark}
{\rm In contrast to $\mathbf{DSt}$, $\mathbf{DPC}$ is not, however, at level 1.  For example, the algebra $\mathbf{EIGHT}$, whose lattice reduct, $\to$ and $'$ are given below, is, in fact, in the subvariety of $\mathbf{DPC}$, defined by: $(x \lor y)' \land (x' \lor y)' \land (x \lor y')'=0$; but it fails to satisfy {\rm(Lev 1)} identity. }

\setlength{\unitlength}{.7cm} \vspace*{.5cm}
\begin{picture} (8,1)  

                         \put(5,1){\circle*{.15}} 
\put(5.2,1){$1$}  

\put(6,0){\circle*{.15}} 
\put(6.2,0.1){$e$}  


\put(4,0){\circle*{.15}}  
\put(3.5,0){$f$}  

\put(5.3,-1.1){$d$}  

\put(3,-1){\circle*{.15}} 
\put(2.5,-1){$c$}
\put(5,-1){\circle*{.15}} 

\put(4,-2){\circle*{.15}}  
\put(3.4, -2.1){$a$}  

\put(6,-2){\circle*{.15}}  
\put(6.2, -2.1){$5$}  

\put(5,-3){\circle*{.15}} 
\put(5.2,-3.3){$0$}  

 \put(5,1){\line(1,-1){1}}
 \put(5,-1){\line(1,1){1}}
\put(4,0){\line(1,1){1}}

\put(4,-2){\line(1,1){1}}
\put(5,-3){\line(1,1){1}}

\put(4,0){\line(1,-1){1}}
\put(5,-1){\line(1,-1){1}}
\put(4,-2){\line(1,-1){1}}

\put(3,-1){\line(1,1){1}}
\put(3,-1){\line(1,-1){1}}

\put(5,-9.5){Figure 3}
\end{picture}

\vspace{1cm}

\begin{tabular}{r|rrrrrrrr}
$'$: & $0$ & $1$ & $e$ & $c$ & $a$ & $b$ & $f$ & $d$\\
\hline
   & $1$ & $0$ & $c$ & $e$ & $1$ & $1$ & $e$ & $1$
\end{tabular} \hspace{.5cm}
\begin{tabular}{r|rrrrrrrr}
$\to$: & $0$ & $1$ & $e$ & $c$ & $a$ & $b$ & $f$ & $d$\\
\hline
    $0$ & $1$ & $0$ & $0$ & $b$ & $b$ & $c$ & $0$ & $0$ \\
    $1$ & $0$ & $1$ & $e$ & $c$ & $a$ & $b$ & $f$ & $d$ \\
    $e$ & $0$ & $1$ & $1$ & $c$ & $c$ & $b$ & $f$ & $f$ \\
    $c$ & $b$ & $c$ & $a$ & $1$ & $e$ & $0$ & $c$ & $a$ \\
    $a$ & $b$ & $c$ & $c$ & $1$ & $1$ & $0$ & $c$ & $c$ \\
    $b$ & $c$ & $b$ & $b$ & $0$ & $0$ & $1$ & $b$ & $b$ \\
    $f$ & $0$ & $1$ & $e$ & $c$ & $a$ & $b$ & $1$ & $e$ \\
    $d$ & $0$ & $1$ & $1$ & $c$ & $c$ & $b$ & $1$ & $1$
\end{tabular} \hspace{.5cm}
\end{remark}

\ \\ \ \\ \ \\ 

\begin{Lemma} \label{DSt_Lemma} 
Let $\mathbf{L} \in \mathbf{DSt}$, and $x,y \in L$.  Then $\mathbf{L}$ satisfies:\\

\indent $(x \lor y)' =(x'^* \lor y'^{*})^*= x' \land y'$. \label{2}   
\end{Lemma}

\begin{proof}
 Since $(x \to y)'^* \leq x \to y$, we have $x \land (x \to y)'^* = x \land (x \to y) \land (x \to y)'^* = x \land y \land (x \to y)'^*$, proving 
 the lemma.
\end{proof}

The following corollary is immediate from Lemma \ref{DSt_Lemma}, 
Theorem \ref{DSt_JID} and   ~\cite[Corollary 8.2(a)]{Sa11}.  

\begin{Corollary} \label{DSt_DISCR}
$ \mathbf{DSt}$ is a discriminator variety of level \rm 1.  \\       
\end{Corollary}
\ \\
\subsection{The variety $\mathbf{DStHC}$}
\ \\ \ \\
Recall that $\mathbf{DStHC}$ is the variety generated by dually Stone Heyting chains.   
It follows from Lemma \ref{DSt_Lemma} 
that $\mathbf{DSt}$ satisfies (B).  The following corollary is, therefore, immediate from Theorem \ref{DSt_JID}(b) and 
Theorem \ref{Theorem_BDQD}. 

\begin{Corollary}\label{Cor_DSt}
Let $\mathbf{L} \in \mathbf{DSt}$ with $|L| > 2$.  Then the following are equivalent:
\begin{enumerate}
\item[(1)] ${\mathbf L}$ is simple,
\item[(2)]  $\mathbf{L}$ is subdirectly irreducible,
\item[(3)] $\mathbf{L}$ satisfies {\rm (SC)}.
\end{enumerate}
\end{Corollary}

We now give an application of Corollary \ref{Cor_DSt}.
For $n \in \mathbb{N}$, let $\mathbf{C_n^{dp}}$ denote the n-element $\mathbf {DStH}$-chain  (= $\mathbf {DPCH}$-chain) and $\mathbf {V(C_n^{dp})}$ denotes the variety generated by
$\mathbf{C_n^{dp}}$.  (Note that $\mathbf {C_3^{dp}}$ was denoted by $\mathbf {L_1^{dp}}$ in \cite{Sa11}.)  

Recall from Corollary \ref{Cor_DStC} that $\mathbf {DPCHC} =\mathbf{DStHC}$.
The following theorem was implicit in \cite[Section 13]{Sa11}.

\begin{Theorem} \label{Subvarlat_DStHC}
The lattice of subvarieties of $\mathbf {DStHC}$ is the following $\omega + 1$-chain:
$$\mathbf{V(C_1^{dp})} < \mathbf{V(C_2^{dp})} < \dots <\mathbf{V(C_n^{dp})}< \dots < \mathbf {DStHC}.$$
\end{Theorem}

\begin{proof}
Let $\mathbf{C^{dp}}$ be a $\mathbf {DStC}$-chain.  
Since $x=1$ or $x'=1$, for every $x \in C$, it is clear that  $\mathbf{C^{dp}}$ satisfies (SC).  On the other hand, let $\mathbf{A} \in \mathbf {DStHC}$ satisfy (SC).
Let $a \in A \setminus \{1\}$.  By Theorem \ref{DSt_JID} (a) we have $a'^* \leq a$; hence by (SC), we get 
$a'^*=0$, implying $a' = 1$, again by Theorem \ref{DSt_JID} (a).  Since each $\mathbf {DStHC}$-chain, 
being a Heyting-chain, satisfies the identity $(x \to y) \lor (y \to x) \approx 1$, it follows that $\mathbf {DStHC}$ 
satisfies it too, implying that any two elements of $\mathbf{A}$ are comparable in $\mathbf{A}$, so $\mathbf {A}$ is 
a $\mathbf {DStH}$-chain.
Thus, $\mathbf {A} \in \mathbf {DStHC}$ is subdirectly irreducible iff $\mathbf {A}$ is a $\mathbf {DStH}$-chain.   Now it is not hard to observe that if an identity fails in an infinite $\mathbf {DStHC}$-chain, then it fails in a finite $\mathbf {DStHC}$-chain.  Thus 
$\mathbf {DStHC}$ is generated by finite $\mathbf {DStHC}$-chains.  Hence, the conclusion follows.         
\end{proof}

A similar argument can be used to prove the following theorem.

\begin{Theorem}
$ \mathbf {DStHC} = \mathbf {DStL}.$  
\end{Theorem}

Note, however, that if we consider $\mathbf{DStC}$-chains with semi-Heyting reducts that are 
not Heyting algebras, the situation gets complicated since the structure of the lattice of subvarieties of $ \mathbf {DStC}$ is quite complex,  as shown by the following class of examples: Let $A$ be a semi-Heyting algebra.  Let $A^e$ be the expansion of A by adding a unary operation $'$ as follows: 
$$x'=0, \text{ if  } x=1, \text{ and  } x'=1, \text{ otherwise}. $$ 
Then it is clear that $A^e$ is a $\mathbf{DSt}$-algebra and is simple.   In particular, if $A$ is a semi-Heyting-chain, then $A^e \in  \mathbf {DStC}$ and is simple.   Furthermore, the number of semi-Heyting chains even for small size is large; for example, there are 160 semi-Heyting chains of size 4. 
It is interesting to observe that $\mathbf{\bar{2}^e} \in \mathbf{DStC} \setminus  \mathbf{DStHC}$, and $ \mathbf{DStHC}$ is only a "small" subvariety of $\mathbf{DStC}$.  These observations suggest that the following problem is of interest:\\

{\bf Problem: Investigate the structure of the lattice of subvarieties of $\mathbf{DStC}$.}

\vspace{1cm}
\section{Subdirectly Irreducible Algebras in $\mathbf {JID_1}$} 

Recall that the variety $\mathbf{JID_1}$ 
is the subvriety of $\mathbf{JID}$ defined by
\begin{center}
\hspace{-4cm} {\rm(Lev $1$)}  \qquad  $x \land x'^*  
\approx (x \land x'^*)^{{'^*}}$.
\end{center}

In this section we give a characterization of subdirectly irreducible (=simple) algebras in the variety $\mathbf{JID_1}$.  
Such a characterization will be obtained as an application of Theorem \ref{Theorem_BDQD}.

The following theorem   
follows immediately
from Theorem \ref{DmsSH} and Theorem \ref{Theorem_BDQD}.
\begin{Theorem} \label{Corollary_4C}
Let $\mathbf{L} \in \mathbf {JID_1}$ with $|L| > 2$.  Then the following are equivalent:
\begin{enumerate}
\item[(1)] ${\mathbf L}$ is simple,
\item[(2)]  $\mathbf{L}$ is subdirectly irreducible,
\item[(3)] $\mathbf{L}$ satisfies {\rm (SC)}. \label{3}
\end{enumerate}
\end{Theorem}

We now wish to characterize the subdirectly irreducible algebras in $\mathbf {JID_1}$.   In  view of the above theorem, it suffices to characterize the algebras in $\mathbf {JID_1}$ satisfying the condition (SC). 

{\bf Unless otherwise stated, in the rest of this section we assume that $\mathbf{L} \in 
\mathbf {JID_1}$ with $|L| \geq 2$ and satisfies
the simplicity condition (SC). }

\begin{Lemma}\label{Lemma_3.232}  Let $a, b \in L$ such that 
$a'=a$.  
Then
 \[ a \lor b \lor b^* = 1.\]   
\end{Lemma}

 \begin{proof}
 From Lemma \ref{L1} (4) and $a'=a$, we have
\begin{equation} \label{EC1}
 (a \lor b) \to a = a \lor  b^*. 
\end{equation}
Now,
\begin{alignat*}{3}
a \lor (a \lor b)'^* &= a' \lor [(a \lor b)' \to 0] && \\   
&= [a' \lor (a \lor b)'] \to (a' \lor 0), && \text{ by  (JID)}\\                
&= a'  \to a'  && \text{ as $a' \geq (a \lor b)'$}\\                                                   
&= 1.  &&                   
\end{alignat*}
 Thus, we have
\begin{equation} \label{EC3}
a \lor (a \lor b)'^* =1. 
\end{equation}
If $a \lor b = 1$, then clearly the lemma is true.  So, we assume that
 $a \lor b \neq 1$.  Then $(a \lor b) \land (a \lor b)'^* = 0$ by (SC), and hence, we have
\begin{alignat*}{2}
a &= a \lor (b \land b^*) \\
&=  (a \lor b) \land (a \lor b^*)\\ 
&= (a \lor b) \land [(a \lor b)'^* \lor (a \lor b^*)] \\
&= (a \lor b) \land [(a \lor b)'^* \lor \{(a \lor b) \to a\}] &\text{by \eqref{EC1}} \\ 
&= (a \lor b) \land [\{(a \lor b) \to a\} \lor (a \lor b)'^*]\\
&= (a \lor b) \land [a \lor (a \lor b)'^*]   &\text{ by Lemma \ref{Lemma_3.70}} (3)\\ 
&= a \lor b  &\text{ by \eqref{EC3}}.   
\end{alignat*}
Hence, $a  \lor b =a$, 
which implies, by \eqref{EC1}, that $a \lor b^*  =1$.  
The conclusion of the lemma is now immediate.  
\end{proof}

\begin{Lemma}\label{Crucial_small_Lemma} 
Let $x \in L \setminus \{1 \}$.  Then\  $x \leq x' $.   
\end{Lemma}

\begin{proof}
 Since $x  \neq 1$, we have $x \land x'^* = 0$ by (SC), from which we get
$(x' \lor x) \land (x' \lor x'^*)= x'$, whence 
$x' \lor x =x'$, as $x' \lor x'^*=1$ by Lemma \ref{L1} (3) proving the lemma. 
\end{proof}

\begin{Lemma}\label{Theorem_height}  
Let 
$|L|>2$ and let $a \in L$ such that $a'=a$.  
 Then the height of $L$ is at most $2$.
\end{Lemma} 

\begin{proof}
Suppose there are $b,c \in L$ such that $0 < b  <  c < 1$.  We wish to arrive at a contradiction.\\
From Lemma \ref{Crucial_small_Lemma} we have $c \leq c'$,  
from which it follows that  
\begin{equation} \label{eqJ}  
b \leq c'.
\end{equation}

{\bf Claim 1}:  $b'=1$.  

Suppose $b' \neq 1$.  Then, by Lemma \ref{Crucial_small_Lemma}, we get $b' \leq b'' \leq b \leq c$; thus $b' \leq c$.  Next, $b \leq c$ implies $c' \leq b'$; and also $c \leq c'$ from Lemma 
\ref{Crucial_small_Lemma}, whence $c \leq b'$.  Thus we conclude that $b'=c$, whence $c' =b'' \leq b$, implying $c'=b$, by \eqref{eqJ}.
Then, in view of 
Lemma \ref{Crucial_small_Lemma}. we have $c \leq c' = b$; thus $c \leq b$,  
which is a contradiction, proving the claim.

 From Lemma \ref{Lemma_3.232}  we have $a \lor b \lor b^*=1$.
Hence, $a' \land b' \land b{^*}' = 0$ by Theorem \ref{Theorem_105},  
implying $a  \land b{^*}'=0$ by {\bf Claim 1} and the hypothesis.    
Thus
\begin{equation} \label{eqL}     
a \land b{^*}' =0.
\end{equation}
Therefore, 
$a \lor b^* \geq a \lor b{^*}'' =1$ as $a'=a$, yielding 
$b \leq a$. 
Hence, again from \eqref{eqL},    
we obtain
\begin{equation} \label{eqN}  
b \land b{^*}' =0.
\end{equation}

{\bf Claim 2}: $b \lor b^*=1$. 

Suppose the claim is false.  Then $b \leq b \lor b^* \leq (b \lor b^*)'$ by Lemma \ref{Crucial_small_Lemma}, 
whence 
$b \leq b' \land b{^*}'$, which implies $b = b \land b' \land b{^*}' = 0$ by  
the equation \eqref{eqN}, 
contrary to 
$b > 0$, proving the claim.

From {\bf Claim 2} and Theorem \ref{Theorem_105} we have  $b' \land b{^*}'=0$,   
Since $b'=1$ by {\bf Claim 1}, it follows that
$b{^*} \geq b{^*}'' = 1$; so $b \leq b^{**} = 0$, contradicting $b > 0$, proving the lemma.    
\end{proof}

\begin{Lemma} \label{Helper} 
For every $x \in \mathbf{L}$, $x=1$ or $ x'=1 $ or $x = x' $.  
\end{Lemma}

\begin{proof}
 Suppose $x \in L$ such that  $x \neq 1$ and $x' \neq 1$.  Then 
 by Lemma \ref{Crucial_small_Lemma}, we have  
$x \leq x'$.
Also, since $x' \neq 1$, we have $x' \leq x'' $, again by Lemma \ref{Crucial_small_Lemma}.   
So, $x = (x \lor x') \land x = (x \lor x') \land (x \lor x'') =x \lor (x' \land x'') = x \lor x'$, so $x \geq x'$, implying $x=x'$, proving the lemma.
\end{proof}

\begin{Lemma}\label{LA} 
Let $a \in L$ such that $a'=a$.  Then $a{^*}'=a^*$.  
\end{Lemma}

\begin{proof} 
First, observe that $a \neq 0$ and $a \neq 1$, since $a = a'$.   
Suppose $a{^*}' \neq a^*$.
The following claims will lead to a contradiction. \\  

{\bf Claim 1}:  $a^*  = a{^*}'' $.

$a \lor a{^*}'' =  a'' \lor a{^*}'' =  (a \lor a{^*})'' = [a' \lor (a' \to 0)]'' = 1$ 
by Lemma 3.4(3).  Hence,  $a \lor a{^*}'' = 1$, implying $a^* \leq a{^*}''$, and so $a^*= a{^*}''$, proving the claim.\\

{\bf Claim 2}:   $ a^*=0. $ 

 We have, by Lemma \ref{Helper},  
  that $a{^*}'=1$ or $a{^*}'' = 1$ or $a{^*}'=a{^*}''$.  So, by {\bf Claim 1}, we get   
 $a^*=a{^*}''=0$ \ or \ $a{^*} = 1$ (as $a^* \geq a{^*}''$)  \ or \ $a{^*}'=a{^*}$.  But, we know, by our assumption, that $a^* \neq a{^*}'$. 
Hence, $a{^*} = 0$ or $a{^*} =1$, which clearly implies  $a^*=0$ or $a=0$.  Since we know that $a \neq 0$, the claim is proved.  

Now, in view of (JID) and {\bf Claim 2}, we have $a= a \lor 0 = a \lor a^*=a' \lor (a \to 0) = (a' \lor a) \to (a' \lor 0)= a \to a = 1$, implying $a =1$, which is a contradiction, proving the lemma.
\end{proof}

\begin {Proposition} \label{Theorem_DQDBSH}  
Let $|L| > 2$.  Suppose there is an $a  \in L$ such that $a'=a$.  Then 
 $\mathbf{L} \in \{\mathbf{D_1, D_2, D_3}\}$,  up to isomorphism.  
\end{Proposition} 

\begin{proof}
In view of Lemma \ref{Theorem_height} and $|L| > 2$, the height of $L$ is  $2$.  Since the lattice reduct of $L$ is distributive, $L$ is either a $3$-element chain or a $4$-element Boolean lattice.   We know from Lemma \ref{LA} that 
$a{^*}'=a^*$.  Thus $a$ and $a^*$ are complementary, implying that the lattice reduct of $\mathbf{L}$
is a $4$-element Boolean lattice; so $\mathbf{L} \models (Bo)$.  Then, from      
Proposition \ref{Prop_DQB} (a)
it follows that $\mathbf{L} \in  \{\mathbf{D_1, D_2, D_3}\}$, up to isomorphism.
\end{proof}

\begin{Proposition} \label{Theorem_DSSH}
Suppose $x' \neq x$, for every $x \in L \setminus \{1 \}$.
Then
 $\mathbf{L} \in \mathbf{DSt}$.  
\end{Proposition}

\begin{proof}    
Let $x \in L \setminus \{1 \}$.  Suppose that  
$x' \neq 1$.  Then,  
arguing as in the proof of {\bf Claim 1} of Lemma \ref{Theorem_height}, we get
$ x \leq x'$ and 
 $ x' \leq x'' \leq x$, which implies $x = x'$, contradicting the hypothesis.  So   
 $x'=1$, which implies $x' \land x'' \approx 0$, 
 Hence $\mathbf{L}$ is a dually Stone semi-Heyting algebra.
\end{proof}

We are now ready to prove our main theorem of this section.

\begin{Theorem}  \label{Th_JID} 
Let $\mathbf{ L} \in \mathbf{DQD}$ with $|L| > 2$. 
Then the following are equivalent:
\begin{enumerate}
 \item[{\rm(a)}] $\mathbf{ L}$ is a subdirectly irreducible algebra in $\mathbf{JID_1}$,
 \item[{\rm(b)}]  $\mathbf{ L}$  is a simple algebra in $\mathbf{JID_1}$, 
 \item[{\rm(c)}]  $\mathbf{ L} \in \mathbf{JID_1}$ such that {\rm (SC)} holds in $\mathbf{ L}$,
 \item[{\rm(d)}] $\mathbf{L}  \in \{\mathbf{D_1, D_2, D_3}\}$, up to isomorphism, or  $\mathbf{L} \in \mathbf{DSt}$ 
and $\mathbf{L}$ satisfies {\rm (SC)}.
\end{enumerate}
\end{Theorem}

\begin{proof} 
In view of Theorem \ref{Corollary_4C}, 
we only need to prove (c) $\Leftrightarrow$ (d).  Now, suppose (d) holds.  Then it is routine to verify that 
$\{\mathbf{D_1, D_2, D_3}\} \subseteq \mathbf{JID_1}$ and $\{\mathbf{D_1, D_2, D_3}\}$ satisfies (SC).  For the second case, it suffices to apply
Theorem \ref{DSt_JID}(c).   
Thus (d) $\Rightarrow$ (c).  Next, suppose (c) is true 
 in $\mathbf{L}$.                      
 We consider two cases: 
 
First, suppose there is an $a \in L$ such that $a'=a$.  Then, by Proposition \ref{Theorem_DQDBSH}, $\mathbf{L} \in \{\mathbf{D_1}, \mathbf{D_2}, \mathbf{D_3}\}$.

Next, suppose 
$\mathbf{L}$ satisfies: 
\begin{equation} \label{eqnE}
\text{For every } x \in L, \ x' \neq x. 
\end{equation} 
Then, using 
Proposition \ref{Theorem_DSSH}, we obtain that $\mathbf{L}$ is dually Stone, which leads us to conclude   
 (c) $\Rightarrow$ (d).
\end{proof}

We have the following important consequence of Theorem \ref{Th_JID}.

\begin{Corollary}
$\mathbf{JID_1} = \mathbf{DSt} 
\lor \mathbf{V(D_1, D_2, D_3)}. $ 
\end{Corollary}

Now, we focus on the subvariety $\mathbf{JIDH_1}$ of $\mathbf{JIDH}$.  Note that the variety of Boolean algebras is the only atom in the lattice of subvarieties of $\mathbf{JIDH_1}$.  For $\mathbf{V}$ a subvariety of $\mathbf{JIDH_1}$, let $\mathcal{L}\mathbf{(V)}$ and $\mathcal{L}^+\mathbf{(V)}$ denote, respectively, the lattice of subvarieties of $\mathbf{V}$ and the lattice of nontrivial subvarieties of 
$\mathbf{V}$.  Let  $\mathbf{1} \oplus \mathbf{L}$ denote the ordinal sum of the trivial lattice $\mathbf{1}$ and a lattice $\mathbf{L} $. 
  
Restricting the semi-Heyting reduct in the above corollary to Heyting algebras, we obtain the following interesting corollary.

\begin{Corollary}  We have
\begin{enumerate}
\item $\mathbf{JIDH_1} = \mathbf{DStH}  \lor \mathbf{V(D_2)}$,
\item $\mathcal{L} \mathbf{(JIDH_1)} \cong  \mathbf{1} \oplus (\mathcal{L}^+ \mathbf{(DStH)} \times \mathbf{2})$.  
\end{enumerate}
\end{Corollary}

The preceding corollary leads to the following open problem.\\

PROBLEM:  Investigate the structure of $\mathcal{L}^+\mathbf{(DStH)}$.

\vspace{1cm}
\section{JI-distributive, dually quasi-De Morgan, linear Semi-Heyting Algebras}  
\medskip

In this section we focus on the 
linear identity:
$$ \text{\hspace{-2in}(L) } \qquad  (x \to y) \lor (y \to x) \approx 1.$$

Let $\mathbf {DQDL}$ [$\mathbf {JIDL}$] denote the subvariety of  $\mathbf {DQD}$ [$\mathbf {JID}$] defined by (L),
and let $\mathbf {JIDLH}$ denote the subvariety of $\mathbf {JIDL}$ consisting of JI-distributive,  dually quasi-De Morgan, linear Heyting algebras.

 The following result (stated in the current terminology) is needed later in this section. 
 
\begin{Proposition} \label{PH} \cite [Lemma 12.1(f)] {Sa11}
Let $\mathbf {L}$ be a linear semi-Heyting algebra.   Then                                     
$\mathbf {L} \models {\rm (H)}$. 
Hence, $\mathbf{JIDL} = \mathbf {JIDLH}$.  
\end{Proposition}

 \begin{Lemma} \label{L4_10}
 Let $\mathbf{L} \in \mathbf{DQDL}$ and let $x,y \in L$.  Then
 \begin{enumerate}
  \item[{\rm(a)}]  $(x \to y) \lor (y \to x)'' =1$,  
 \item[{\rm(b)}] $x\  \leq \ y \lor (y \to x)'' $.   
\end{enumerate}  
\end{Lemma}

\begin{proof}
$(x \to y) \lor (y \to x)'' \geq (x \to y)''  \lor (y \to x)'' 
                                 =     [(x \to y) \lor (y \to x)]'' 
                                 = 1$  by (L), proving (a).
Using Lemma \ref{2.2} (\ref{xii}) and (a), we get $x \land [y \lor (y \to x)''] = x \land [(x\to y) \lor (y \to x)'']     = x$, implying (b).   
\end{proof}

Note that the algebra $\mathbf {SIX}$ described earlier in Section 3 is actually an algebra in
$\mathbf {JIDL}$.  Hence $\mathbf {JIDL}$ does not satisfy (Lev 1);         
but $\mathbf {JIDL}$ is at level 2, in view of Theorem \ref{DmsSH}. 

In this section, our goal is to present, as an application of Theorem \ref{Th_JID}, an explicit description of subdirectly irreducible (= simple) algebras in the variety $\mathbf{JIDL_1}$ of JI-distributive, dually quasi-De Morgan, linear semi-Heyting algebras of level 1.

Recall that $\mathbf{DPCC}=\mathbf{DStC}$ (and $\mathbf{DPCHC}=\mathbf{DStHC}$).
So, we use these names 
 interchangeably.

\begin{Lemma} \label{Level1}
$\mathbf{DPCC} \models {\rm(Lev 1)}$.  
\end{Lemma}

\begin{proof}
Let $\mathbf{L}$ be a  $\mathbf{DPC}$-chain and 
let $x \in L$.  
Since $x, x'$ are comparable, we have  $x \lor x'  = x$ or $x \lor x' = x'$, implying 
$x=1 \ or \ x'=1$, as $x \lor x'=1.$ 
Then it is easy to see that (Lev 1) holds in $\mathbf{L}$, and hence in $\mathbf{DPCC}$.  
\end{proof}

\begin{Proposition}\label{Half_Join_Th}
 $\mathbf {DPCHC} \,\lor  \,\mathbf {V(D_2)} \, \subseteq \, \mathbf {JIDL_1}$.   
\end{Proposition}

\begin{proof}
It follows from Lemma \ref{DPCSHC_JID},  and Lemma \ref{Level1} that $\mathbf {DPCHC}$ satisfies (JID) and (Lev 1), and it is easy to see that $\mathbf{DPCHC} \models {\rm(L)}$.   
Also, it is routine to verify that (JID), (L)  and (Lev 1) hold in $\mathbf{D_2}$.  
\end{proof}

Next, we wish to prove the reverse inclusion of the above corollary.\\

{\bf Unless otherwise stated, in the rest of this section we assume that $\mathbf{L} \in 
\mathbf {JIDL_1}$ with $|L| > 2$ and satisfies (SC).}

\begin{Lemma} \label{Helper_Two}
Let $x,y \in L$.  Then, $x \lor y  \neq 1$ implies $ x \leq y'.$      
\end{Lemma}

\begin{proof}
Let $x \lor y \neq 1$; then, since $y' \lor (x \lor y)'^* \geq y' \lor y'^* =1$ by Lemma \ref{L1} (3),
we get $x=x \land (x \lor y) \land [y' \lor (x \lor y)'^*]= x \land [\{(x \lor y) \land y'\} \lor \{(x \lor y) \land (x \lor y)'^*\}] = x \land y'$ using (SC), whence  
$x  \leq y'$.  
\end{proof}

\begin{Lemma} \label{L3} 
Let $a,b \in L$ such that 
$a' \neq a$,  $a \neq 1$, 
and  $a \not\leq b$. 
Then  $(a \to b)''=0$.    
\end{Lemma}

\begin{proof}

We claim that $a  \not \leq (a \to b)''$.  For, suppose $a  \leq (a \to b)''$; then $a = a \land (a \to b)'' \leq b$ by Lemma \ref{2.2} (\ref{xix}), 
implying $a \leq b$, which is a contradiction to 
 the hypothesis $a \not \leq b$.   
 Hence  $a  \not \leq (a \to b)''$.  Then  $a \lor (a \to b)' =1$ by (the contrapositive of) Lemma \ref{Helper_Two},   
whence   $a'' \lor (a \to b)''' = 1$, from which we conclude $(a \to b)'' = 0$, as $a'=1$ by Lemma \ref{Helper}.  
\end{proof}

We can now give an explicit description of subdirectly irreducible (=simple) algebras in $\mathbf{JIDL_1}$.  

\begin{Theorem}  \label{Th} 
Let $\mathbf{ L} \in \mathbf{DQD}$ with $|L| > 2$. 
Then the following are equivalent:
\begin{enumerate}
 \item[(1)] $\mathbf{ L}$ is a subdirectly irreducible algebra in $\mathbf{JIDL_1}$, 
 \item[(2)] $\mathbf{ L}$  is a simple algebra in $\mathbf{JIDL_1}$,
 \item[(3)] $\mathbf{ L} \in \mathbf{JIDL_1}$ such that {\rm (SC)} holds in $\mathbf{ L}$, 
 \item[(4)] $\mathbf{L}  \cong \mathbf{D_2}$, or $\mathbf{ L}$ is a $\mathbf{DStH}$-chain.
\end{enumerate}
\end{Theorem}

\begin{proof} 

(1) $ \Leftrightarrow  {\rm(2)}  \Leftrightarrow \rm{(3)}$ follow from Theorem \ref{Th_JID}.
So we need to prove (3) $\Rightarrow$ (4) $\Rightarrow$ (3).     
Suppose (3) holds.  
Then, by Theorem \ref{Th_JID},  either $\mathbf{L} \in \{\mathbf{D_1}, \mathbf{D_2}, \mathbf{D_3}\}$, or $\mathbf{L} \in \mathbf{DSt}$ 
and satisfies {\rm (SC)}.   
In the former case, since $\mathbf{L} \models {\rm(L)}$, it follows from Proposition \ref{PH} that $\mathbf{L}\models {\rm(H)}$.  Hence $\mathbf{L} \cong \mathbf{D_2}$.  Next, we assume $\mathbf{L} \in \mathbf{DSt}$  and satisfies {\rm (SC)}.  
Since $\mathbf{L} \models {\rm(L)}$ by hypothesis, we get, by Proposition \ref{PH}, that $\mathbf{L} \in \mathbf{DStH}$. 
So, we need only prove that $\mathbf{L}$ is a chain.  Let $a, b \in L \setminus \{1\}$ such that 
$a \not \leq b$.  Since $\mathbf{L} \models {\rm (DSt)}$,   
it is clear that $a' \neq a$.      
Then, from Lemma  \ref{L4_10}(b),     
we have that $b \leq a \lor (a \to b)''$, which, in view of Lemma \ref{L3}, implies             
  $b \leq a$.  
 Hence, the lattice reduct of $\mathbf{L}$ is a chain,  
 and so, (3) $\Rightarrow$ (4).  
 Finally, assume (4).   
 First, if $\mathbf{L} \cong \mathbf{D_2}$, then clearly (3) holds.   
 Next, suppose $\mathbf{L}$ is a $\mathbf{DStH}$-chain.  Then, $x' \leq x''$ or $x'' \leq x'$, implying $x' \land x''=x'$ or $x' \land x'' =x''$.  Hence, by (DSt), we get $x'=0$ or $x'=1$, from which it is easy to see that $\mathbf{L}$ satisfies (SC).  Now,  from Theorem \ref{DSt_JID}, we conclude that $\mathbf{L} \in \mathbf{JID_1}$.  Also, it is well known that Heyting chains satisfy (L).  Thus, $\mathbf{L} \in \mathbf{JID_1}$ and $\mathbf{L}$ satisfies (SC), implying (3).   
\end{proof}

The following corollary is immediate from Theorem \ref{Th}.
\begin{Corollary} \label{C1}
$\mathbf{JIDL_1} 
=  \mathbf{DStHC}\  \lor\  \mathbf{V(D_2)}$. 
\end{Corollary}

\begin{proof}
Use Proposition \ref{Half_Join_Th} and Theorem \ref{Th}.
\end{proof}

We would like to mention here that the attempt to solve the problem of axiomatization of $ \mathbf{DStHC}\  \lor \ \mathbf{V(D_2)}$ led to the results of this paper,  with Corollary \ref{C1} yielding a solution thereof.

\vspace{1cm}
\section{More Consequences of Theorem \ref{Th}}

In this section we present some more consequences of Theorem \ref{Th}.  

As mentioned earlier, the axiomatizations of the variety $\mathbf {DPCHC}$ 
 and all of its subvarieties were given in \cite{Sa11}.  
 
 The following corollary is immediate from Corollary \ref{C1} and Theorem \ref{Subvarlat_DStHC}. 

\begin{Corollary} 
\begin{thlist}
\item[1]
  $\mathcal{L}\mathbf{(JIDL_1)} \cong \mathbf{1} \oplus [\mathbf{(\omega + 1)} \times \mathbf{2}]$. \item[2]
   $\mathbf{JIDL_1}$ and $\mathbf{DStHC}$ are the only two elements of infinite height in the lattice $\mathcal{L}\mathbf{(JIDL_1)}$.
\item[3] 
    $\mathbf{V} \in  \mathcal{L^+} \mathbf{(JIDL_1)}$ is of finite height iff $\mathbf{V}$ is either $\mathbf{V(D_2)}$ or $\mathbf{V(C_n^{dp})}$, for some $n \in \mathbb{N} \setminus \{1\}$, or $\mathbf{V(C_m^{dp})}  \, \lor \, \mathbf{V(D_2)}$, for some $m \in \mathbb{N} \setminus \{1\}$.     
\end{thlist}
\end{Corollary}

In Corollaries \ref{cor_6_2}-\ref{cor_6_5}, we give equational bases to all subvarieties of $\mathbf {JIDL_1}$. 
  
\begin{Corollary} \label{cor_6_2}
The variety  $\mathbf {DStHC}$ is defined, modulo $\mathbf {JIDL_1}$, by 
\begin{enumerate}
\item[ ] $x \lor x' \approx 1$.
\end{enumerate}
\end{Corollary}

\begin{proof} 
Observe that $\mathbf {DStHC} \models x \lor x' \approx 1$, but $\mathbf {V(D_2)} \not\models x \lor x' \approx 1$, and then apply 
Theorem \ref{Th}.
\end{proof}

The variety $\mathbf {V(D_2)}$ was axiomatized in \cite{Sa11}.  Here is a new one.
\begin{Corollary}\label{cor_6_3}
The variety $\mathbf {V(D_2)}$ is defined, modulo $\mathbf {JIDL_1}$, by 
\begin{enumerate}
\item[ ] $x'' \approx x$.
\end{enumerate}
\end{Corollary}

\begin{proof} 
Observe that $\mathbf {DStHC}\not \models x'' \approx x$, but $\mathbf {V(D_2)} \models x'' \approx x$, and then use Theorem \ref{Th}.
\end{proof}
 
\begin{Corollary} \label{cor_6_4}
Let $n \geq 2$.  
The variety $\mathbf {V(C_n^{dp})}\, \lor \, \mathbf {V(D_2)}$ is defined, modulo $\mathbf {JIDL_1}$, by 
\begin{enumerate}
\item[{\rm(C$_n$)}] $x_1  \lor  x_2  \lor \cdots \lor x_n  \lor  (x_1 \to x_2)  \lor (x_2 \to x_3) \lor \cdots 
\lor (x_{n-1} \to x_n)=1.$
\end{enumerate}
\end{Corollary}

\begin{proof}
Let $\mathbf{C_n^{dp}}$ be the $\mathbf{DPC}$-chain such that\\
 $C_n^{dp} = \{0, a_1, a_2,  \dots, a_{n-2}, 1\}$, where
 $0 < a_1 < a_2 <  \cdots < a_{n-2} < 1$.  We now prove that $\mathbf{C_n^{dp}} \models {\rm (C_n)}$.
Let $\langle c_1, c_2, \dots , c_n \rangle \in C_n^{dp}$ be an arbitrary assignment in $C_n^{dp}$ for the variables such that $c_i$ is the value of $x_i$, for $i=1, \cdots, n$.
If $c_i \leq c_{i+1}$ for some $i$, then $c_i \to c_{i+1} = 1$,
as $\mathbf{L}$ has a Heyting reduct, and hence, the identity holds in $\mathbf{C_n^{dp}}$. So, we assume
that $c_i > c_{i+1}$, for $i = 1,2, \cdots , n$. Then, $c_1 = 1$ since $|C_n^{dp}| = n$, implying that 
 (C$_n$) holds in $\mathbf{C_n^{dp}}$.  Also, observe that $\mathbf {V(D_2)} \models {\rm (C_n)}$. Next, suppose that $\mathbf{V}$ is the subvariety of 
 $\mathbf{JIDL_1}$ satisfying (C$_n$). Then, by Corollary \ref{Discr}, $\mathbf{V}$ is a discriminator variety. So, let $\mathbf{L}$ be a simple algebra in $\mathbf{V}$. Then, it follows from Corollary 
 \ref{C1} (or Theorem \ref{Th}) that the semi-Heyting reduct of $\mathbf{L}$ is a Heyting chain or $\mathbf{L} \cong \mathbf{D_2}$.  Suppose that the semi-Heyting reduct of $\mathbf{L}$ is a Heyting chain, and 
assume $|L| > n$, then there exist $b_1,b_2 \cdots, b_{n-1} \in L$ such that $0 < b_1 < \cdots, <
 b_{n-1} < 1$.  Since $\mathbf{L} \models$ (C$_n$), we can assign $\langle b_{n-1}, b_{n-2}, \cdots, b_1, 0 \rangle$ for $\langle x_1, x_2, \cdots, x_{n-1}, x_n \rangle$.  Then, $b_{n-1}  \lor (b_{n-1} \to b_{n-2}) \lor  \cdots, \lor (b_1 \to 0) = 1$, yielding 
 \[b_{n-1} \lor b_{n-2} \lor \cdots \lor b_1 \lor 0 = 1, \] 
implying that 
$b_{n-1} = 1$, which is a contradiction. Thus we have $|L| \leq n$, and we conclude that 
$\mathbf{V} = \mathbf{V(\mathbf{C_n^{dp}})}$.  Next, suppose $\mathbf{L} \cong \mathbf{D_2}$; then 
clearly $\mathbf{V} \cong \mathbf{V(D_2)}$, completing the proof.
\end{proof}

\begin{Corollary}\label{cor_6_5}
The variety $\mathbf {V(C_n^{dp})}$ is defined, modulo $\mathbf {JIDL_1}$, by 
\begin{enumerate}
\item[{\rm(1)}] $x \lor x' \approx 1$,
\item[{\rm(2)}] $x_1  \lor  x_2  \lor \cdots \lor x_n  \lor  (x_1 \to x_2)  \lor (x_2 \to x_3) \lor \cdots 
\lor (x_{n-1} \to x_n)=1.$
\end{enumerate}
\end{Corollary}

For a different base for $\mathbf {V(C_n^{dp})}$, see \cite{Sa11}.  Regularity was studied in \cite{Sa11}, \cite{Sa14},  \cite{Sa14a} and \cite{Sa15}.  Here is another use of it. 

\begin{Corollary}
The variety $\mathbf {V(C_3^{dp})}\,\lor \, \mathbf {V(D_2)}$ is defined, modulo $\mathbf {JIDL_1}$, by 
\begin{enumerate}
\item[] $x \land x^+ \leq y \lor y^*$   {\rm (Regularity)}, {\rm where $x^+ :=x'{^*}'$.}
\end{enumerate}
It is also defined, modulo $\mathbf {JIDL_1}$, by 
\begin{enumerate}
\item[] $x \land x'  \leq y \lor y^*$.
\end{enumerate}
\end{Corollary}

The variety $\mathbf{V(C_3^{dp})}$ is axiomatized in 
\cite{Sa11}.  Here is another axiomatization for it.
\begin{Corollary}
The variety $\mathbf{V(C_3^{dp})}$ is defined, modulo $\mathbf{JIDL_1}$, by
\begin{itemize}
\item[(1)] $x \land x^+\leq y \lor y^*$  {\rm (Regularity)},
\item[(2)] $x'=x^+$.
\end{itemize}
\end{Corollary}

We now examine the Amalgamation Property for subvarities of the variety $\mathbf{DStHC}$.   
For this purpose, we need the following 
lemma whose proof is straightforward.  

We use ``$\leq$''  to abbreviate ``is a subalgebra of'' in the next lemma.

Recall from \cite{Sa11} that the proper, nontrivial subvarieties of $\mathbf{DStHC}$ are precisely the subvarieties  of the form 
$\mathbf{V(C_n^{dp})}$, for $n \in \mathbb{N}$.

\begin{Lemma} Let $m, n \in \mathbb{N}$.  Then
\begin{enumerate} 
\item[]  $\mathbf{C_m^{dp}}\, \leq \,\mathbf{C_n^{dp}}$, for $m \leq n$.
\end{enumerate}
\end{Lemma}

\begin{Corollary}
Every subvariety of $\mathbf{DStHC}$ has Amalgamation Property.
\end{Corollary}

\begin{proof} 

It follows from Corollary \ref{DSt_DISCR}  
that $\mathbf{DStHC}$ is a discriminator variety; and hence has CEP.  Also, from Theorem \ref{Th} we obtain that every subalgebra of each subdirectly irreducible (= simple) algebra in $\mathbf{DStHC}$ is subdirectly irreducible.  Let $\mathbf{V}$ be a subvariety of $\mathbf{DStHC}$. Then,using a result from \cite{GrLa71} that we need only consider an amalgam ($\mathbf{A}$: $\mathbf{B}$, $\mathbf{C}$), where $\mathbf{A}$, $\mathbf{B}$, $\mathbf{C}$  are simple in $\mathbf{V}$ and $\mathbf{A}$ a subalgebra of $\mathbf{B}$ and $\mathbf{C}$.  First, suppose $\mathbf{V} =\mathbf{V(C_n^{dp})}$ for some $n$.  
Then $\mathbf{B}$ and $\mathbf{C}$ are $\mathbf{DStHC}$-chains.  Then, in view of the preceding lemma, it is clear that the amalgam ($\mathbf{A}$: $\mathbf{B}$, $\mathbf{C}$) can be amalgamated in $\mathbf{V}$.   
Next, suppose  $\mathbf{V} = \mathbf{DStHC}$,    
then it is clear that the amalgamation can be achieved as in the previous case. 
\end{proof}

We conclude this section with the following remark: Since every subvariety $\mathbf{V}$ of $\mathbf{DStHC}$ has Congruence Extension Property and Amalgamation Property, it follows from Banachewski \cite{Ba70} that 
all subvarieties of $\mathbf{DStHC} $ have enough injectives (see \cite{Ba70} for the definition of this notion).\\

\small

\ \\
\ \ \ \ \ \ \ \ \\
Department of Mathematics\\
State University of New York\\
New Paltz, NY 12561\\
\
\\
sankapph@newpaltz.edu\\
\ \

\end{document}